\documentclass[11pt]{article}
\usepackage{latexsym, graphicx, epsfig, amsmath,amssymb,amsthm, bm}
\usepackage{geometry}
\usepackage{caption, subcaption}
\usepackage{epstopdf}
\usepackage{hyperref}
\usepackage{enumitem}
\usepackage{comment}
\usepackage{color}

\title{Predictability of Observables of Dynamical Systems\footnote{This work is partially supported by AFOSR FA9550-24-1-0237.}}
\author{Xinyu Liu and Dongbin Xiu\footnote{ Department of Mathematics, The Ohio State University, Columbus, OH, USA. Emails: \texttt{liu.12165@osu.edu, xiu.16@osu.edu.}}}
\date{}

\newtheorem{theorem}{Theorem}[section]
\newtheorem{lemma}[theorem]{Lemma}
\newtheorem{proposition}[theorem]{Proposition}
\newtheorem{corollary}[theorem]{Corollary}
\theoremstyle{definition}
\newtheorem{definition}[theorem]{Definition}
\theoremstyle{remark}
\newtheorem{remark}[theorem]{Remark}

\begin{document}

\maketitle

\begin{abstract}
We study the evolution of observables of dynamical systems. For linear systems, we show that observables satisfy a closed differential equation whose minimal order is determined by the dynamical system and observation operator. This yields a minimal order closure and an equivalent discrete delay representation of the observable dynamics. For nonlinear systems we introduce the notion of \emph{diminishing ambiguity}, which provides a framework under which the instantaneous observable dynamics can be approximately determined from sufficiently long output history, resulting in delay differential equation representation. These results clarify when observable dynamics can be inferred from past history without knowledge of the dynamical system and its full state.
\end{abstract}

\section{Introduction}

Consider an autonomous dynamical system
\begin{equation} \label{u}
    \dot u(t) = A(u(t)), \qquad u(t)\in\mathbb{R}^n,
  \end{equation}
  where we assume the system is well defined such that
  solution existence and uniqueness hold. 
  
In many practical situations, the system and its full state $u(t)$ are unavailable. Instead, one observes a set of observables which are functions of the full state. Let us define observables
\begin{equation} \label{y}
y(t) = b(u(t)), \qquad y(t)\in\mathbb{R}^m, \quad m\leq n.
\end{equation}
A few examples of observables are
\begin{itemize}
    \item Average: $b=\frac{1}{n}(1,\cdots,1)^T$, $y=b^T u = \frac{1}{n}\sum_{i=1}^n u_i$;
    \item Component selection: $b=e_k$, $y=b^T u = u_k$, $k=1,\dots,n$;
    \item $\ell_2$-norm: $b=u^Tu = \|u\|^2_2$;
    \item Maximum-norm: $b=\max_{1\leq k\leq n} |u_k|=\|u\|_{\infty}$.
\end{itemize}

We study the following question: {\em when do the observables $y(t)$ satisfy a closed dynamical evolution law that depends only on $y(t)$?}
This paper develops a structural study for this problem.

\paragraph{Related Work}
The problem studied in this paper is related to the classical embedding theory of Takens \cite{Takens1981} and its extensions \cite{SauerYorkeCasdagli1991,Robinson2011}. Takens’ embedding theorem shows that, under generic conditions, the state of a dynamical system can be reconstructed from delayed observations of a scalar observable. Delay coordinates also play an important role in modern data-driven modeling and Koopman operator methods \cite{ArbabiMezic2017,BruntonProctorKutz2016}.

More recently, numerical models for observable dynamics have been
constructed directly from observable history without reconstructing
the underlying state. One example is the Flow Map Learning (FML) framework
\cite{FuChangXiu_JMLMC20, ChurchillXiu_FML2023}, which is motivated by a finite memory approximation of the Mori-Zwanzig formalism \cite{mori1965, zwanzig1973}.

\paragraph{Contribution of this work}
The objective of the present work is fundamentally different from the (Takens') embedding perspective, which requires the assumption that the observable map $b$ is injective on the state space in order to reconstruct the underlying state $u$.
In this paper, we do not assume $b$ is injective and allow multiple (even infinitely many) states $u$ to map to same observable $y$. Such is often the case in practical applications. Consequently, we do not seek to reconstruct the state $u$. Instead, we study the mathematical
conditions under which the future evolution of an observable $y(t)$ can be determined directly from its past history. 

Compared to the FML work, this paper provides a more fundamental mathematical study, without requiring the finite-memory assumption used in FML. In Section \ref{sec:linear}, we
show that for linear systems the observable dynamics admits a finite
delay representation analytically. The result is based on the finite dimensional structure of the observable trajectories. 

For nonlinear systems, the finite-dimensional structure in observable dynamics is absent in general. In Section \ref{sec:nonlinear}, we introduce the notion of {\em diminishing ambiguity}, which establishes a condition under which the observable dynamics can be approximated by its past history without requiring knowledge of the underlying state.

\section{Preliminaries}

Throughout this paper, we shall
fix a compact forward-invariant set $K\subset \mathbb{R}^n$ to represent the regime of interest for the solution $u(t)$. For example, this can be a trapping region, a compact invariant set, or a neighborhood of an attractor. 
All trajectories of $u(t)$ are assumed to remain in $K$ for the time intervals under discussion.

For the observable, we employ the following notation, commonly used in delay differential equation literature.
For $h\in(0,\infty]$ and a trajectory $y(\cdot)$, define its $h$-history segment
\begin{equation} \label{yt}
y_{t,h}(\theta) := y(t+\theta), \qquad \theta\in[-h,0],
\end{equation}
with $y_{t,\infty}$ understood when $h=\infty$.

\begin{definition}[Admissible histories] \label{def:history}
Let $\mathcal{H}_\infty$ be the set of functions $\phi:(-\infty,0]\to\mathbb{R}^m$ such that there exists a trajectory
$u(\cdot)$ of the ODE with $u(t)\in K$ for all $t\le 0$ and $\phi(\theta)=b(u(\theta))$ for all $\theta\le 0$.
For $h<\infty$, define $\mathcal{H}_h := \{\phi|_{[-h,0]} : \phi\in\mathcal{H}_\infty\}$.
\end{definition}

In the following, we shall first discuss the case of linear dynamical systems with linear observables, where many properties can be (relatively) easier to understand. We will then discuss general nonlinear systems.

\section{Linear Systems with Linear Observations}
\label{sec:linear}

Let us first consider linear dynamical system with linear observables,
\begin{equation} \label{lin-lin}
\left\{
\begin{split}
    &\dot u(t)=Au(t), \qquad A\in\mathbb{R}^{n\times n}, \\
    & y(t)=Bu(t), \qquad B\in\mathbb{R}^{m\times n}, \quad m\leq n,
\end{split}
\right.
\end{equation}
where $A$ and $B$ are assumed to be full rank. Then $u$ and $y$ are $C^\infty$, and the derivatives satisfy
\begin{equation} \label{derivative}
    y^{(k)}(t)=BA^k u(t), \qquad \forall k\ge 0.
\end{equation}

\subsection{Krylov Subspace and Minimal Order Closure}

Let us consider the observable Krylov subspace for the system \eqref{lin-lin}. It is defined as
\begin{equation} \label{Vk}
\mathcal V_k = \mathrm{range}\big[B^\top, A^\top B^\top, \dots, (A^\top)^k B^\top\big]
\subseteq\mathbb{R}^n, \qquad k\geq 0,
\end{equation}
where the bracket denotes horizontal concatenation.

Equivalently,
\begin{equation}
\mathcal V_k
= \mathrm{span}\Big\{ (A^\top)^j B^\top v \;:\; j=0,\dots,k,\ v\in\mathbb R^m \Big\}
\subseteq \mathbb R^n,
\end{equation}
where each $(A^\top)^j B^\top$ is an $n\times m$ matrix and $\mathcal V_k$ is a subspace of $\mathbb R^n$ spanned by its columns. Obviously,
\[
\mathcal V_0 \subseteq \mathcal V_1 \subseteq \cdots\subseteq \mathbb{R}^n,
\]
and 
\[
\dim \mathcal{V}_0 \leq \dim \mathcal{V}_1\cdots\leq n.
\]
We then immediately obtain the following result.
\begin{lemma}{(Stabilization index)}
There exists a stabilization index $r\le n-1$ such that $\dim \mathcal V_{r-1}<\dim \mathcal V_r=\dim \mathcal V_{r+1}\leq n$, where $\dim \mathcal{V}_{-1}=0.$
\end{lemma}

\begin{theorem}{(Minimal order closure)} \label{thm:k-ODE}
Let $r$ be the stabilization index.
Then there exist matrices $C_0,\dots,C_r\in\mathbb R^{m\times m}$ such that
\begin{equation} \label{BA}
BA^{r+1} = \sum_{k=0}^{r} C_k BA^k.
\end{equation}
Equivalently, by defining the (monic) matrix polynomial
\begin{equation} \label{Q}
Q(\lambda):=\lambda^{r+1}I_m-\sum_{k=0}^{r}C_k\,\lambda^k,
\end{equation}
we have
\begin{equation} \label{BQ}
B\,Q(A)=0.
\end{equation}
Consequently, $y(t)$ satisfies the $(r+1)$-th order linear vector ODE
\begin{equation} \label{kth-ODE}
y^{(r+1)}(t) = \sum_{k=0}^{r} C_k y^{(k)}(t).
\end{equation}
Moreover, $r+1\le n$, and it is minimal with this property.
\end{theorem}

\begin{proof}
The stabilization $\mathcal V_r=\mathcal V_{r+1}$ implies
\[
(A^\top)^{r+1}B^\top \in \mathcal V_r.
\]
Hence there exist matrices $D_0,\dots,D_r\in\mathbb R^{m\times m}$ such that
\begin{equation}
(A^\top)^{r+1}B^\top
= \sum_{k=0}^r (A^\top)^k B^\top D_k^\top.
\end{equation}
Transposing yields \eqref{BA} with $C_k=D_k$.
Applying this to $u(t)$ and using the derivative ODE \eqref{derivative} gives \eqref{kth-ODE}.
Minimality follows from minimality of $r$ in the Krylov stabilization.
\end{proof}

\begin{remark}[Relation to Cayley--Hamilton]
The Cayley--Hamilton theorem gives $p_A(A)=0$ for the characteristic polynomial $p_A$ of degree $n$.
Multiplying by $B$ yields $Bp_A(A)=0$, so a closure always exists with order at most $n$.
Theorem~\ref{thm:k-ODE} identifies the minimal order $r+1 \le n$
determined by the Krylov sequence.
\end{remark}

\begin{remark}
    If $B\in\mathbb{R}^{n\times n}$ and is full rank, then the observables $y$ are essentially the state variables $u$, as $u=B^{-1}y$. It is trivial to see that the stabilization index of \eqref{Vk} is $r=0$ and $y$ follows the 1st-order ODE $\dot{y}=BAB^{-1}y$.
\end{remark}

\begin{corollary}[Invariance under differentiation]
\label{cor:invariance}
Let $\mathcal S$ denote the space of observable trajectories associated with the
linear system \emph{(4)}. Then $\mathcal S$ is invariant under differentiation.
That is, if $y\in\mathcal S$, then $\dot y\in\mathcal S$.
\end{corollary}

\begin{proof}
By Theorem~3.2, every observable trajectory $y\in\mathcal S$ satisfies
the $(r+1)$-th ODE \eqref{kth-ODE}. Differentiate this relation with respect to $t$ and set $z(t):=\dot y(t)$. Then $z$ satisfies the same linear differential equation as $y$.
Therefore $z=\dot y\in\mathcal S$.
\end{proof}

\subsection{Delay Representations}

The minimal order closure result from Theorem \ref{thm:k-ODE} gives us a way to establish delay representations of the observable dynamics. 

\subsubsection{Scalar Observable}

Consider scalar observable case, $m=1$ in \eqref{lin-lin}, i.e.,
\begin{equation} \label{scalar-lin}
    \dot u(t)=Au(t), \qquad y=b^T u,
\end{equation}
where $b$ is a nonzero vector of length $n$.

\begin{theorem}[Discrete delay representation] \label{thm:discrete1}
For scalar observable ($m=1$) system \eqref{scalar-lin}, let $r$ be the stabilization index of the observable Krylov space \eqref{Vk},
$$
\mathcal V_k = \mathrm{range}\left[b, A^\top b, \dots, (A^\top)^k b\right]
\subseteq \mathbb{R}^n, \qquad k\geq 0.
$$
Then for almost all choices of distinct delays 
\begin{equation} \label{tau}
    0\leq \tau_1<\cdots<\tau_{r+1},
\end{equation}
there exist unique
coefficients $w_1, \cdots, w_{r+1}\in\mathbb{R}$ 
such that every observable trajectory satisfies
\begin{equation} \label{discrete}
\dot y(t) = \sum_{k=1}^{r+1} w_k y(t-\tau_k), \qquad \forall t.
\end{equation}
\end{theorem}

\begin{proof}
Let 
$$
\mathcal{S} = \left\{ y(\cdot): \dot u(t)=Au(t),~ y=b^T u\right\}.
$$
From Theorem \ref{thm:k-ODE}, $y(t)$ satisfies a minimal scalar ODE of order $r+1$. Then $\mathcal{S}$ is a finite-dimensional linear space of analytic functions with  
\[\dim \mathcal{S} = r+1.
   \]
Choose a basis $\psi_1,\dots,\psi_{r+1}$ of $\mathcal S$.
For delays $\tau_1,\dots,\tau_{r+1}$, we define evaluation matrix
\[
M(\tau_1,\dots,\tau_{r+1})
=
\begin{bmatrix}
\psi_1(-\tau_1) & \cdots & \psi_{r+1}(-\tau_1)\\
\vdots & \ddots & \vdots\\
\psi_1(-\tau_{r+1}) & \cdots & \psi_{r+1}(-\tau_{r+1})
\end{bmatrix},
\]
which is a $(r+1)\times(r+1)$ square matrix whose determinant depends
analytically on the delays.

We first show that $\det M$ is not identically zero.
Indeed, if $\det M\equiv 0$, then for every choice of distinct delays
there would exist a nontrivial linear combination
\[
\phi(t)=\sum_{j=1}^{r+1} c_j \psi_j(t)
\]
that vanishes at $r+1$ arbitrarily prescribed distinct points.
Since $\phi$ is analytic, this would force $\phi\equiv 0$,
contradicting linear independence of the basis.
Therefore, $\det M$ is a nontrivial analytic function of the delays.
Its zero set is a proper analytic subset of $\mathbb{R}^{r+1}$, whose measure is zero.
Therefore, $\det M\neq 0$ for almost all choices of distinct delays.

Let us fix such a choice of distinct delays. Then $M$ is invertible, and for each $t$
\[
\begin{bmatrix}
y(t-\tau_1)\\ \vdots\\ y(t-\tau_{r+1})
\end{bmatrix}
= M \, c(t),
\]
where $c(t)$ are the coordinates of $y$ in the chosen basis and can be uniquely determined by inverting $M$.
Therefore, $y(t-\tau_1),\dots, y(t-\tau_{r+1})$ is a basis of $\mathcal{S}$. The expression \eqref{discrete} then follows from the differentiation invariance property of $\mathcal{S}$ of Corollary \ref{cor:invariance}.

\end{proof}

\begin{corollary}[Continuous delay operator] \label{cor:cont1}
Assume the same condition of Theorem \ref{thm:discrete1} holds. 
There exists $h>0$ and a finite signed atomic measure $\mu$ supported on $[0,h]$
such that every observable trajectory satisfies
\begin{equation} \label{kernel}
    \dot y(t) = \int_0^h y(t-s)\, d\mu(s).
\end{equation}

Equivalently, there exists a bounded linear functional
\begin{equation} \label{L}
    L : C([-h,0]) \to \mathbb R
\end{equation}
such that
\begin{equation} \label{cont1d}
    \dot y(t) = L(y_{t,h}), 
\end{equation}
where $y_{t,h}(\theta)=y(t+\theta),\ \theta\in[-h,0]$ as defined in \eqref{yt}.

\end{corollary}
\begin{proof}
From the discrete delay representation \eqref{discrete}, we define
\[
h := \max_{1\le k\le r+1} \tau_k,
\qquad
\mu := \sum_{k=1}^{r+1} w_k\,\delta_{\tau_k},
\]
where $\delta_{\tau_k}$ denotes the Dirac measure at $\tau_k$.
Then
\[
\sum_{k=1}^{r+1} w_k\, y(t-\tau_k)
=
\int_0^h y(t-s)\, d\mu(s),
\]
and \eqref{kernel} follows.

Define the operator
\[
L(\phi)=\int_0^h \phi(-s)\, d\mu(s).
\]
Since $\mu$ is a finite signed measure, $L$ is a bounded linear functional
on $C([-h,0])$. This yields \eqref{cont1d}.
\end{proof}

\subsubsection{Vector Observables}

We now consider the more general case of vector observables $m>1$ for \eqref{lin-lin}. This is a structural generalization of the scalar result in Theorem \ref{thm:discrete1}.
\begin{theorem} [Vector discrete delay representation] \label{thm:linear}
    For the linear system with linear observables \eqref{lin-lin}, let $r$ be the stabilization index of the observable Krylov space \eqref{Vk}.
Then for almost all choices of distinct delays
\[
0\leq\tau_1<\cdots<\tau_{r+1},
\]
there exist unique matrices $W_1,\dots,W_{r+1}\in\mathbb{R}^{m\times m}$
such that every observable trajectory satisfies
\begin{equation} \label{discrete-m}
    \dot y(t)=\sum_{k=1}^{r+1} W_k\,y(t-\tau_k),
\qquad \forall t.
\end{equation}

\end{theorem}

\begin{proof}
    See Appendix \ref{app:linear}.
\end{proof}

This theorem establishes existence of a vector discrete delay representation. 
The continuous delay representation, similar to Corollary \ref{cor:cont1} for scalar observable, follows naturally.
\begin{corollary}[Vector continuous delay representation] \label{cor:linear}
There exists $h>0$ and a finite signed matrix-valued atomic measure $\mu$
supported on $[0,h]$ such that
\begin{equation}
    \dot y(t)=\int_0^h d\mu(s)\, y(t-s).
\end{equation}
Equivalently, 
\begin{equation}
    \dot y(t)=L(y_t),
\end{equation}
where $L:C([-h,0];\mathbb R^m)\to\mathbb R^m$ is a bounded linear operator.
\end{corollary}

\begin{remark} \label{rem:linear}
The results in this section establish that for the linear system \eqref{lin-lin}, the observables do satisfy a closed dynamical evolution law that depends only on the history. 
Theorem \ref{thm:linear} states that the observable history over a set of finite number of discrete delays can uniquely determine its dynamics. We shall refer to this result as {\em discrete uniqueness}. Corollary \ref{cor:linear} states that observable history over a finite interval is also sufficient to determine its dynamics --- a result we shall refer to as {\em interval uniqueness}. Note that for linear system \eqref{lin-lin}, interval uniqueness is ``overdetermined" in the sense a history interval contains an infinite number of delays. Since the observable dynamics is finite dimensional, according to Theorem \ref{thm:k-ODE}, discrete uniqueness is sufficient. This shall not be the case for nonlinear system, which is the topic of the next section.
\end{remark}

\section{Nonlinear Systems and Diminishing Ambiguity}
\label{sec:nonlinear}

We now return to the original, generally nonlinear, system 
\begin{equation} \label{nonlinear}
    \dot{u} = A(u), \qquad y = b(u), \qquad u\in\mathbb{R}^n, \quad y\in\mathbb{R}^m, \quad m\leq n.
\end{equation}

In order for the observable $y$ to form a dynamical system of itself, it is necessary that its past history can uniquely determine its derivative for the immediate future. Without loss of generality, let us consider $t=0$. The question is then:
\begin{quote}
    {\em Given observable history $y_{0,h}(\theta) = y(\theta)$, $\theta\in [-h,0]$, whether $\dot{y}(0^+)$ is uniquely determined.}
\end{quote}
The explicit use of the derivative at $0^+$ 
acknowledges the fact that the dynamics of $y(t)$ may not be smooth and is entirely driven by the dynamics of $u(t)$.
In fact, consider the Lie derivative of $y(t)$ along the flow of $u$,
\begin{equation} \label{lee_d}
\ell(u) = Db(u)A(u),
\end{equation}
we have
\[
\dot y(0^+) = \ell(u(0)).
\]

In contrast to the linear case, observable trajectories typically do not lie in a finite-dimensional space. As a result, finite collections of discrete delays are generally insufficient to determine the observable dynamics. Instead, one must consider dependence on the full past history.
We distinguish the following two questions:
(i) whether the infinite observable history uniquely determines the instantaneous observable dynamics;
and (ii) whether the observable dynamics can be approximated using only a finite history segment.

The first question is related to determinism of the observable dynamics, while the second concerns finite-memory approximability. We formalize these notions below.

\subsection{$T$-history Determinism}

\begin{definition}[$T$-history determinism] \label{def:history}
    We say the observable dynamics of \eqref{nonlinear} is deterministic with respect to history of length $T>0$ if for any two trajectories $u_1(\cdot), u_2(\cdot)\in K$ satisfying
\[
b(u_1(t)) = b(u_2(t)), \quad \forall t \in[-T,0],
\]
one has
\[
\dot{y}_1(0^+) = \dot{y}_2(0^+).
\]
\end{definition}

\begin{proposition}
    Assume admissible trajectories of \eqref{nonlinear} satisfy $u \in C^1$ and $b \in C^1(\mathbb{R}^n;\mathbb{R}^m)$. Then the observable dynamics is deterministic with respect to history of any $T>0$.

\end{proposition}

\begin{proof}
    Since $u\in C^1$ and $b\in C^1$, 
    then $y(t)=b(u(t))$ is $C^1$. Equality of $y$ on an interval $[-T,0]$ implies equality of its derivative at $t=0$. Hence $\ell(u_1(0))=\ell(u_2(0))$.
\end{proof}

This result shows that $T$-history determinism is not restrictive in smooth finite-dimensional systems. The main difficulty lies instead in approximating the observable dynamics when the determinism does not hold.

\subsection{Diminishing Ambiguity}

To quantify finite-memory predictability, we introduce the notion of diminishing ambiguity (DA).

\begin{definition}[Diminishing Ambiguity] \label{def:DA}
The system \eqref{nonlinear} is said to have diminishing ambiguity (DA) if 
there exist constants $\alpha>0$ and $C>0$ such that for any $T>0$ and any two trajectories
$u_1(\cdot),u_2(\cdot)\in K$ satisfying
\begin{equation*}
y_1(t) = y_2(t) \qquad \forall t\in[-T,0]
\end{equation*}
one has
\begin{equation} \label{DA}
\|\dot{y}_1(0^+)-\dot{y}_2(0^+)\| \le C e^{-\alpha T}.
\end{equation}
\end{definition}

 This condition quantifies how rapidly the ambiguity in the observable tendency decays as the observable history length increases. The special case of $T=\infty$ leads to infinite-history determinism from Definition \ref{def:history}.

\begin{theorem}[Existence of observable dynamics]
Assume the system \eqref{nonlinear} has diminishing ambiguity (DA). Then,
\begin{itemize}
    \item[(i)]
There exists a well-defined operator $L_\infty:\mathcal{H}_\infty\to\mathbb{R}^m$ such that
\begin{equation} \label{y0plus}
\dot y(0^+) = L_\infty(y_{0,\infty}).
\end{equation}
Or, more generally,
\begin{equation} \label{ddq_inf}
\dot y(t^+) = L_\infty(y_{t,\infty}).
\end{equation}
\item[(ii)]
For each $h>0$ there exists a family of maps $L_h:\mathcal{H}_h\to\mathbb{R}^m$, defined up to $O(e^{-\alpha h})$ ambiguity, such that for any admissible trajectory
\begin{equation}
\|\dot y(0^+)-L_h(y_{0,h})\|\le C e^{-\alpha h}.
\end{equation}

Consequently, we have finite-memory approximation
\begin{equation} \label{dde_h}
\dot y(t^+)=L_h(y_{t,h})+r_h(t), \qquad \|r_h(t)\|\le C e^{-\alpha h}.
\end{equation}

\end{itemize}

\end{theorem}

\begin{proof} (i)
Fix $\phi\in\mathcal{H}_\infty$.
If $u_1,u_2$ generate $\phi$, then letting $T\to\infty$ in DA~\eqref{DA} yields
$\dot{y}_1(0^+)=\dot{y}_2(0^+)$.
Define $L_\infty(\phi)$ to be this common value and
\eqref{y0plus} follows.

(ii) Fix $\phi_h \in H_h$. Choose any trajectory $u$ consistent with $\phi_h$ and define
\[
L_h(\phi_h) := \ell(u(0)).
\]
If $u_1,u_2$ are two such trajectories, then DA with $T=h$ implies
\[
\|\ell(u_1(0)) - \ell(u_2(0))\| \le C e^{-\alpha h}.
\]
Therefore $L_h$ is well defined up to an ambiguity of order $e^{-\alpha h}$, and for any trajectory generating $\phi_h$,
\[
\|\dot{y}(0^+) - L_h(\phi_h)\| \le C e^{-\alpha h}.
\]
\end{proof}
The operator $L_h$ is thus an approximate, history-dependent closure for the observable dynamics. Therefore,
DA provides a quantitative condition under which finite observable history yields accurate prediction of observable dynamics.
In particular, it characterizes when nonlinear observable dynamics admit effective finite-memory representations, even though exact finite-dimensional closures generally do not exist.

\section{Conclusion}

We studied the predictability of observable dynamics in partially observed dynamical systems. 
For linear systems, we showed that observables admit a minimal-order closure determined by a Krylov subspace, leading to equivalent delay representations. 

For nonlinear systems, we distinguished between infinite-history determinism and finite-memory predictability. 
While the former follows from smoothness, the latter requires additional structure. To this end, we introduce the notion of diminishing ambiguity (DA), which provides a quantitative characterization under which observable dynamics can be approximated from finite history.

These results clarify the limits of observable-based modeling and provide a mathematical framework for finite-memory representations. Future work includes identifying verifiable conditions for diminishing ambiguity in specific classes of systems and characterizing minimal memory structures for nonlinear dynamics.

\appendix

\section{Proof for Theorem \ref{thm:linear}} \label{app:linear}

\begin{proof}
Define the augmented state
\[
Y(t):=\big(y(t),\dot y(t),\dots,y^{(r)}(t)\big)\in\mathbb{R}^{m(r+1)}.
\]
Then $Y(t)$ satisfies the linear system
\[
\dot Y(t)=\mathcal A\,Y(t),
\]
where $\mathcal A\in\mathbb{R}^{m(r+1)\times m(r+1)}$ is the block companion matrix
\[
\mathcal A=
\begin{bmatrix}
0 & I & 0 & \cdots & 0\\
0 & 0 & I & \cdots & 0\\
\vdots & \vdots & \vdots & \ddots & \vdots\\
0 & 0 & 0 & \cdots & I\\
C_0 & C_1 & C_2 & \cdots & C_r
\end{bmatrix}.
\]

Let
\[
P_0 := [\, I \;\; 0 \;\; \cdots \;\; 0 \,] \in \mathbb{R}^{m\times m(r+1)},
\qquad
P_1 := [\, 0 \;\; I \;\; 0 \;\; \cdots \;\; 0 \,],
\]
so that
\[
y(t)=P_0 Y(t), \qquad \dot y(t)=P_1 Y(t).
\]
For $\tau\ge 0$, define
\[
F(\tau):=P_0 e^{-\mathcal A \tau} \in \mathbb{R}^{m\times m(r+1)}.
\]
Then
\[
y(t-\tau)=F(\tau)\,Y(t).
\]

First, we show that the row span of the family $\{F(\tau):\tau\ge 0\}$ is
$\mathbb{R}^{m(r+1)}$.

Suppose $v\in\mathbb{R}^{m(r+1)}$ satisfies
\[
F(\tau)v=0 \qquad \forall \tau\ge 0.
\]
Then
\[
P_0 e^{-\mathcal A \tau} v = 0 \qquad \forall \tau\ge 0.
\]
Since $F(\tau)$ is analytic in $\tau$, repeated differentiation is valid. Upon differentiating it at $\tau=0$ repeatedly, we obtain
\[
P_0 \mathcal A^k v = 0, \qquad k=0,1,\dots,r.
\]
The stacked matrix
\[
\begin{bmatrix}
P_0\\
P_0\mathcal A\\
\vdots\\
P_0\mathcal A^r
\end{bmatrix}
\]
is invertible, hence $v=0$.
Therefore the row span of $\{F(\tau)\}$ is all of $\mathbb{R}^{m(r+1)}$.

Next, we construct delays $\tau_1,\dots,\tau_{r+1}$ inductively.

Since the total row span is full, there exists $\tau_1$ such that the $m$
rows of $F(\tau_1)$ are linearly independent.

Assume $\tau_1,\dots,\tau_k$ have been chosen so that the rows of
\[
\begin{bmatrix}
F(\tau_1)\\ \vdots\\ F(\tau_k)
\end{bmatrix}
\]
are linearly independent, giving dimension $mk$.
If for every $\tau$ all rows of $F(\tau)$ lie in this span, then the total
row span of $\{F(\tau)\}$ would have dimension at most $mk<m(r+1)$,
contradicting the earlier result. Hence there exists $\tau_{k+1}$ whose rows add new
independent directions.

Proceeding inductively yields distinct delays
$\tau_1,\dots,\tau_{r+1}$ such that the stacked matrix
\[
M(\tau_1,\dots,\tau_{r+1})
:=
\begin{bmatrix}
F(\tau_1)\\ \vdots\\ F(\tau_{r+1})
\end{bmatrix}
\in \mathbb{R}^{m(r+1)\times m(r+1)}
\]
is invertible.

\medskip

Finally, from
\[
\begin{bmatrix}
y(t-\tau_1)\\ \vdots\\ y(t-\tau_{r+1})
\end{bmatrix}
=
M(\tau_1,\dots,\tau_{r+1})\,Y(t),
\]
the invertibility of $M$ gives us
\[
Y(t)=M^{-1}
\begin{bmatrix}
y(t-\tau_1)\\ \vdots\\ y(t-\tau_{r+1})
\end{bmatrix}.
\]
Hence
\[
\dot y(t)
=
P_1 Y(t)
=
P_1 M^{-1}
\begin{bmatrix}
y(t-\tau_1)\\ \vdots\\ y(t-\tau_{r+1})
\end{bmatrix}.
\]
Defining $W_k$ as the corresponding block rows of $P_1M^{-1}$ yields
\[
\dot y(t)=\sum_{k=1}^{r+1} W_k\,y(t-\tau_k).
\]

Lastly, the determinant $\det M(\tau_1,\dots,\tau_{r+1})$ is an analytic function
of the delays and is not identically zero. Therefore it is nonzero for
almost all choices of distinct delays, which completes the proof.

\end{proof}

\bibliographystyle{plain}
\bibliography{reference}

\begin{thebibliography}{1}

\bibitem{ArbabiMezic2017}
Hassan Arbabi and Igor Mezic.
\newblock Ergodic theory, dynamic mode decomposition, and computation of spectral properties of the koopman operator.
\newblock {\em SIAM J. Appl. Dyn. Syst.}, 16(4):2096--2126, 2017.

\bibitem{BruntonProctorKutz2016}
Steven~L. Brunton, Joshua~L. Proctor, and J.~Nathan Kutz.
\newblock Discovering governing equations from data by sparse identification of nonlinear dynamical systems.
\newblock {\em Proc. Natl. Acad. Sci. U.S.A.}, 113(15):3932--3937, 2016.

\bibitem{ChurchillXiu_FML2023}
Victor Churchill and Dongbin Xiu.
\newblock Flow map learning for unknown dynamical systems: Overview, implementation, and benchmarks.
\newblock {\em J. Mach. Learn. Model. Comput.}, 4(2):173--201, 2023.

\bibitem{FuChangXiu_JMLMC20}
Xiaohan Fu, Lo-Bin Chang, and Dongbin Xiu.
\newblock Learning reduced systems via deep neural networks with memory.
\newblock {\em J. Mach. Learn. Model. Comput.}, 1(2):97--118, 2020.

\bibitem{mori1965}
Hazime Mori.
\newblock Transport, collective motion, and {B}rownian motion.
\newblock {\em Prog. Theor. Phys.}, 33(3):423--455, 1965.

\bibitem{Robinson2011}
James~C. Robinson.
\newblock {\em Dimensions, Embeddings, and Attractors}.
\newblock Cambridge University Press, 2011.

\bibitem{SauerYorkeCasdagli1991}
Tim Sauer, James~A. Yorke, and Martin Casdagli.
\newblock Embedology.
\newblock {\em J. Stat. Phys.}, 65:579--616, 1991.

\bibitem{Takens1981}
Floris Takens.
\newblock Detecting strange attractors in turbulence.
\newblock In {\em Dynamical Systems and Turbulence, Warwick 1980}, volume 898 of {\em Lecture Notes in Mathematics}, pages 366--381. Springer, 1981.

\bibitem{zwanzig1973}
Robert Zwanzig.
\newblock Nonlinear generalized {Langevin} equations.
\newblock {\em J. Stat. Phys.}, 9(3):215--220, 1973.

\end{thebibliography}

\end{document}